\documentclass[12pt]{amsart}

\usepackage[utf8]{inputenc}
\usepackage[english]{babel}
\usepackage{amsmath}
\usepackage{amsfonts}
\usepackage{amssymb}
\usepackage{amsthm}
\usepackage{enumerate}
\usepackage{pifont}
\usepackage{subcaption}
\usepackage{cleveref}
\usepackage{todonotes}
\usepackage{tabto}
\usepackage{dsfont}

\usepackage{color}

\usepackage{tikz}
\usetikzlibrary{backgrounds,calc,decorations.pathmorphing}
\tikzset{every path/.style={thick}}

\usepackage[centering,
			textwidth=15cm,
			textheight=22cm,
			top=3.5cm,
			footskip=40pt,
			marginparwidth=2.9cm
			]{geometry}

\numberwithin{equation}{section}

\theoremstyle{plain}
\newtheorem{thm}{Theorem}[section]
\newtheorem{prop}[thm]{Proposition}
\newtheorem{lem}[thm]{Lemma}
\newtheorem{cor}[thm]{Corollary}

\theoremstyle{definition}
\newtheorem{defi}[thm]{Definition}
\newtheorem{ex}[thm]{Example}
\newtheorem{question}{Question}

\theoremstyle{remark}

\newcommand{\R}{\mathbb{R}}
\newcommand{\dotcup}{\hspace{.22em}\ensuremath{\mathaccent\cdot\cup}\hspace{.22em}}
\newcommand{\conv}{\textup{conv}}
\newcommand{\cone}{\textup{cone}}
\newcommand{\CutP}{\textsc{Cut}^\square}
\newcommand{\CutC}{\textsc{Cut}}
\newcommand{\mypar}{\subsection*}

\newcommand{\cmark}{\ding{51}}
\newcommand{\xmark}{\ding{55}}

\usepackage{xspace}
\newcommand{\MC}{\textsc{MaxCut}\xspace}

\title{Cut Polytopes of Minor-free Graphs}
\author[M. Chimani]{Markus Chimani}
\author[M. Juhnke-Kubitzke]{Martina Juhnke-Kubitzke}
\author[A. Nover]{Alexander Nover}
\author[T. Römer]{Tim Römer}
\address{School of Mathematics/Computer Science, University of Osnabrück, Germany}
\email{\{markus.chimani,juhnke-kubitzke,anover,troemer\}@uni-osnabrueck.de}
\date{}

\begin{document}

\begin{abstract}
The cut polytope of a graph $G$ is the convex hull of the indicator vectors of all cuts in $G$ and is closely related to the \MC problem. We give the facet-description of cut polytopes of $K_{3,3}$-minor-free graphs and introduce an algorithm solving \MC on those graphs, which only requires the running time of planar \MC. Moreover, starting a systematic geometric study of cut polytopes, we classify graphs admitting a simple or simplicial cut polytope.
\end{abstract}

\maketitle

\section{Introduction}\label{Intro}
The problem of finding a maximum cut in a weighted graph, called \MC problem, is well-known in combinatorial optimization, and one of Karp's original 21 NP-complete problems~\cite{karp}.
The research on \MC is driven by a variety of applications ranging from mathematical problems like $\ell^1$-embeddability \cite{ApplicationsCutPolyhedra1} over quantum mechanics \cite{SpinGlas,ApplicationsCutPolyhedra2} to design of electronic circuits \cite{CircuitLayout}.
An overview of applications is given in \cite{ApplicationsCutPolyhedra1,ApplicationsCutPolyhedra2}.

Formally, considering a graph $G=(V,E)$ with edge weights $c_e$, \MC is the problem of finding a node subset $S\subseteq V$ that 
maximizes $\sum_{e \in \delta(S)}c_e$, where $\delta(S)=\{e \in E: |e \cap S|=1\}$. 
The \emph{cut polytope} $\CutP(G)$ is defined as the convex hull of the indicator vectors of \emph{cuts} $\delta(S)$, for all $S\subseteq V$, given by 
$$x^{\delta(S)}_e= \begin{cases}	1, &\text{ if } e \in \delta(S);\\
							0, &\text{ else}.			\end{cases}$$

Although \MC is NP-complete on general graphs, there are some classes of graphs on which polynomial algorithms are known. In \cite{AlgorithmPlanar1,AlgorithmPlanar2} it was shown that \MC can be solved in polynomial time for unweighted planar graphs. This result can be extended to the weighted case \cite{MaxCutPlanar2,MaxCutPlanar1}.

By Kuratowski's Theorem \cite{kuratowski}, a graph is planar if and only if it contains no $K_5$- or $K_{3,3}$-subdivision. As an extension of this, Wagner \cite{WagnerPlanar} proved that a graph is planar if and only if it contains no $K_5$- or $K_{3,3}$-minor.

Using Wagner's result, Barahona  \cite{graphsnotcontractible} introduced a polynomial-time algorithm solving \MC on $K_5$-minor-free graphs in $\mathcal{O}(n^4)$ time. This was generalized by Kaminski \cite{cutsandcontainment} by proving that \MC can be solved in $\mathcal{O}(n^4)$ time on \mbox{$H$-minor-free} graphs, for an arbitrary graph $H$ that admits a drawing with exactly one crossing. 
An extension of the class of $K_5$-minor-free graphs was given by Grötschel and Pulleyblank by introducing \emph{weakly bipartite} graphs \cite{WeaklyBipartite,WeaklyBipartiteK5}. By definition these are the graphs, whose bipartite subgraph polytope is completely described by certain cycle- and edge-inequalities (see \Cref{Preliminaries}).
Moreover, they proved that for positive edge-weights, \MC can be solved in polynomial time one these graphs by using linear programming.
In contrast to these results, \MC is NP-complete on $K_6$-minor-free graphs \cite{graphsnotcontractible}.

Considering cut polytopes, it is particularly interesting to find their \emph{linear description}, i.e., their facet-defining inequalities. If there is a linear description of polynomial size in the input, this gives a polynomial algorithm for \MC. Even though it is unlikely to find such a description for arbitrary graphs, a better understanding of cut polytopes is expected to improve algorithmic results.

Although cut polytopes of complete graphs have been intensively studied (see, e.g., \cite{geometryofcuts}), we are far from a good understanding of these objects, especially for $K_n$, $n \geq 9$.
Even much less is known for cut polytopes of arbitrary graphs.
The latter were considered, e.g., by Barahona and Mahjoub \cite{onthecutpolytope}. As an additional result to the polynomial algorithm on $K_5$-minor-free graphs, they determined all facets of cut polytopes of those graphs.

Not too long ago, Sturmfels and Sullivant \cite{ToricGeometyCutsAndSplits} established a new connection between the study of cut polytopes and commutative algebra, as well as algebraic geometry, by considering related toric varieties. In particular, they conjectured that the cut polytope of a graph is normal if and only if the graph is $K_5$-minor-free.
Among others, the research on these toric varieties and associated cut algebras has been pursued by Engström \cite{CutIdealsK4MinorFree}, Ohsugi \cite{NormalityOfCutPolytopes,GorensteinCutPolytopes}, and Römer and Saeedi Madani \cite{RetractsCutAlgebra}.

It turns out that not much is known about the polyhedral structure of cut polytopes as objects in discrete geometry. We expect new insights in the study of \MC by considering cut polytopes of graphs not containing a specific minor.

\mypar{Our contribution and organization of this paper}
In \Cref{Preliminaries}, we recall basic definitions on graphs and polytopes, and summarize known results on cut polytopes.

In \Cref{K33-minor-free Graphs}, we consider $K_{3,3}$-minor-free graphs. Complementing the results on $K_5$-minor-free graphs, we provide the full linear description of cut polytopes of $K_{3,3}$-minor-free graphs.

Moreover, we give an algorithm solving \MC on $K_{3,3}$-minor-free graphs, requiring only the running time for \MC on planar graphs. This is somewhat surprising, as $K_5$-minor-free graphs admit an easier linear description, while we achieve a better running time for \MC on $K_{3,3}$-minor-free graphs.

Starting the investigation of geometric properties of cut polytopes, in \Cref{Simple and Simplicial Cut Polytopes} we completely characterize graphs that provide a simple or simplicial cut polytope. In particular, it turns out that graphs providing a simple cut polytope are precisely the $C_4$-minor-free graphs. The simplicial case can only occur for finitely many graphs. 

\section{Preliminaries}\label{Preliminaries}
In this section we provide some basic background on graphs and polytopes. Then, we recapitulate some known results on cut polytopes. For notation and results related to graphs we refer to \cite{diestel}, for those related to polytopes to \cite{BrunsGubeladze,ziegler}.

\mypar{Graphs}
We only consider undirected graphs. A graph is \emph{simple}, if it does neither have parallel edges connecting the same two nodes, nor self-loops.
Unless specified otherwise, we only consider simple graphs that contain no isolated nodes in the following. For $k \in \mathbb{N}$, let $[k]=\{1, \dots , k\}$.
Given a graph $G=(V,E)$ we also write $V(G)$ and $E(G)$ for its set of nodes $V$ and its set of edges $E$, respectively.
For $v,w \in V(G)$, let $vw=\{v,w\}$ be the edge between $v$ and $w$. Two nodes $v$ and $w$ are \emph{adjacent} if $vw \in E(G)$. 

A \emph{path} of length $k$ is a sequence of edges $e_1, \dots, e_k$ with $e_i =v_{i-1} v_i$ such that $v_i \neq v_j$ for $0 \leq i < j \leq k$.
Such a sequence with $v_0 = v_k$ is a \emph{cycle}; a cycle of length 3 is a \emph{triangle}.
A graph $H$ is a \emph{subgraph} of $G$, denoted by $H\subseteq G$, if $G$ contains (a subgraph isomorphic to) $H$.
Given a subset $W \subseteq V$, the subgraph \emph{induced} by $W$ is the graph $G[W]=(W, \{uv \in E: u,v \in W \})$.
If an induced subgraph forms a cycle, this is an \emph{induced cycle} and thus \emph{chordless}.
A graph $G$ is \emph{chordal}, if every induced cycle in $G$ has length~$3$.
Maximal planar graphs are \emph{triangulations}.
We fix the following notations for some special classes of graphs:
$C_n$ for the cycle of length~$n;$ $K_n$ for the complete graph on $n$ nodes; $K_{n,m}$ for the complete bipartite graph on $n$ and $m$ nodes per partition set. 

$G$ is a $H$-\emph{subdivision}, if $G$ is obtained from $H$ by replacing edges by internally node-disjoint paths.
The graph $G-e$ is obtained from $G$ by \emph{deleting} the edge $e$. 
The graph $G/e$ is obtained from $G$ by \emph{contracting} edge $e=vw$, i.e., the nodes $v$ and $w$ are identified, and we delete the arising self-loop and merge parallel edges. $G$ contains an $H$-\emph{minor}, if $H$ can be obtained from $G$ by contracting and deleting edges. Otherwise $G$ is \emph{$H$-minor-free}.

$G$ is \emph{$k$-connected} if for each pair of nodes $v,w \in V(G)$ there exist $k$ internally node-disjoint paths from $v$ to $w$.
In particular, $1$-connected graphs are called \emph{connected}.
If $G$ is connected but not $2$-connected, there exists some \emph{cut-node} $v \in V(G)$ such that $G-v=G[V \setminus \{v\}]$ is disconnected.

For two graphs $G$ and $H$, their \emph{union} $G \cup H=(V(G) \cup V(H), E(G) \cup E(H))$ is disjoint if their node sets are; in this case we write $G \dotcup H$.
Assume two graphs $G$, $H$ contain $K_k$ as a subgraph, for some $k \in \mathbb{N}_{>0}$. The \emph{$k$-sum} (or \emph{clique-sum}) is obtained by taking the union of $G$ and $H$, identifying the $K_k$ subgraphs and possibly also removing edges contained in this specific $K_k$. A $k$-sum is \emph{strict}, if no edges are removed. We denote the strict $k$-sum of $G$ and $H$ by $G \oplus _k H$.

By Kuratowski's (Wagner's) Theorem~\cite{kuratowski,WagnerPlanar}, a graph is planar if and only if it contains no $K_5$- or $K_{3,3}$-subdivision (minor, respectively).
Given a $K_5$-subdivision $H$ contained in $G$ we call the nodes of degree 4 in $H$ \emph{Kuratowski nodes}. The paths in $H$ between these nodes are \emph{Kuratowski paths}.

\mypar{Cut Polytopes}
A \emph{polytope} $P$ is the convex hull of finitely many points in $\R^d$. 
The \emph{dimension} of $P$ is the dimension of its affine hull.
A linear inequality $a^Tx \leq b_0$ is a \emph{valid inequality} for $P$ if it is satisfied by all points $x \in P$.  It is \emph{homogenous} if $b_0=0$. A (proper) \emph{face} of $P$ is a (non-empty) set of the form $P \cap \{x \in \R^d: a^Tx = b_0\}$ for some valid inequality $a^T x \leq b_0$ with $a\neq 0$. Each face is itself a polytope. The faces of dimension~$0$ and dimension $\dim(P)-1$ are \emph{vertices} and \emph{facets}, respectively.
For polytopes $P \subseteq \R^n$ and $Q \subseteq \R^m$ we define the \emph{product} 
\mbox{$P \times Q = \{(x,y) \in \R^{n+m}: x \in P, y \in Q\}$}. It is a polytope with $\dim(P \times Q)=\dim(P) + \dim(Q)$, and the proper faces of $P \times Q$ are given by products of proper faces of $P$ and proper faces of $Q$.

If $P\cap\{x \in \R^d: a^Tx = b_0\}$ is a facet of $P$, the inequality $a^Tx \leq b_0$ is \emph{facet-defining}.
Each polytope can be represented as the bounded intersection of finitely many closed half-spaces, i.e., $P$ admits a \emph{linear description} $P=\{x \in \R^d: Ax \leq b\}$ for some matrix $A$ and some vector $b$.
This is given, e.g., by taking the system of all facet-defining inequalities.
A \emph{simplex} of dimension $d$ is the convex hull of $d+1$ affinely independent points.
A $d$-dimensional polytope $P$ is \emph{simple} if each vertex of $P$ is contained in exactly $d$ facets; the polytope is \emph{simplicial} if each facet of $P$ is a simplex.
\medskip

Given a graph $G=(V,E)$ and a subset $S \subseteq V$, the set $\delta=\delta(S)=\delta(V\setminus S)=\{e \in E: |e \cap S| =1 \}$ is a \emph{cut} in $G$.
If $G$ is connected, this gives $2^{|V|-1}$ pairwise different cuts. To each cut $\delta$ in $G$ we associate its indicator vector $x^\delta \in \R^E$ given by
$$x^\delta_e= \begin{cases}	1, & \text{ if } e \in \delta;\\
							0, & \text{ else.}			\end{cases}$$
The \emph{cut polytope} of $G$ is defined as their convex hull
$$\CutP(G)=\conv(\{ x^\delta: \delta \text{ is a cut in } G \})\subseteq \R^E,$$
and has dimension $\dim(\CutP(G))=|E(G)|$, see, e.g., \cite[p.344]{bipartitesubgraph}. For disconnected graphs $G_1\dotcup G_2$ we have
\begin{equation}\label{CutPolytopeDisjoint}
\CutP(G_1 \dotcup G_2)=\CutP(G_1) \times \CutP(G_2).
\end{equation}
Similarly, we may consider clique sums.
As many classes of graphs can be described in terms of these, it is reasonable to study their effect on cut polytopes.
\begin{thm}[see {\cite[Theorem 3.1.]{graphsnotcontractible}}]\label{facetsksum}
Let $G=G_1 \oplus_k G_2$ be a strict $k$-sum with $k \in \{1,2,3\}$. Then the facet-defining inequalities of $G$ are given by taking all facet-defining inequalities of $G_1$ and $G_2$ and identifying the variables of common edges. In particular, it holds that 
\begin{equation}\label{CutPolytope1Sum}
\CutP(G_1 \oplus_1 G_2)=\CutP(G_1) \times \CutP(G_2).
\end{equation}
\end{thm}

Any automorphism $\phi$ of a graph $G$ gives rise to a map on cuts. Thus, $\phi$ induces a permutation on the vertices of $\CutP(G)$ by mapping $x^\delta$ to $x ^{\phi(\delta)}$, which yields a symmetry of $\CutP(G)$.
Another symmetry of cut polytopes is given by 
\emph{switching}:
\begin{lem}[Switching Lemma, see {\cite[Corollary 2.9.]{onthecutpolytope}}]
Let $G=(V,E)$ be a graph and $a^T x \leq b$ be a facet-defining inequality for $\CutP(G)$. Let $W \subseteq V$, and define 
$b'= b - \sum_{e \in \delta(W)}a_e$, and 
$a'_e=(-1)^{\mathds{1}[e \in \delta(W)]}\cdot a_e$ 
for all $e\in E$.
Then $(a')^T x \leq b'$ defines a facet of $\CutP(G)$.
\end{lem}
On the level of cuts, switching in $\delta(W)$ is induced by the map $\delta \mapsto \delta \triangle \delta(W)=(\delta \cup \delta(W))\setminus(\delta \cap \delta(W))$.
Switching a facet-defining inequality by a cut corresponding to a vertex of this facet gives a homogeneous facet-defining inequality. Thus, all symmetry classes of facets of $\CutP(G)$ contain facets of the \emph{cut cone}
$\CutC(G)= \cone(\{x \in \CutP(G)\})\subseteq \R^E$.
Hence, it suffices to understand the facets of cut cones to understand the facets of cut polytopes.

Since $\CutP(G)$ is contained in the unit cube, the inequalities $0 \leq x_e \leq 1$ are valid.  
Given a cut $\delta$ and a cycle $C$ in $G$, the number of edges in $\delta\cap C$ clearly is even. These observations give rise to the following \emph{edge}- and \emph{cycle-inequalities}:
\begin{thm}[see {\cite[Section 3]{onthecutpolytope}}]\label{facetsK5minorfree}
The valid inequalities $0 \leq x_e \leq 1$ define facets of $\CutP(G)$ if and only if $e$ does not belong to a triangle.
The valid inequalities
$$\sum_{f \in F}x_f -\sum_{e \in E(C) \setminus F} x_e \leq |F| -1,\qquad  \text{ for all cycles $C\subseteq G$, $F \subseteq E(C)$ with $|F|$ odd}$$
define facets if and only if $C$ is chordless.

Moreover, a graph $G$ is $K_5$-minor-free if and only if $\CutP(G)$ is defined completely by the cycle- and edge inequalities.
\end{thm}

In particular, for each triangle $\Delta$ with $E(\Delta)=\{e,f,g\}$ the following \emph{metric inequalities} (up to permuting the edges) are facet-defining for $\CutP(G)$:
$$x_e +x_f +x_g \leq 2\qquad\text{and}\qquad
x_e -x_f -x_g \leq 0.$$

As a generalization of metric inequalities we get \emph{hypermetric} inequalities by considering the complete graph instead of triangles, see \cite[Theorem~2.4.]{onthecutpolytope}. An example for these is given by the hypermetric inequality of $K_5$ in Inequality~\eqref{K5eq}.
All above facets correspond to complete subgraphs or, in the case of cycle inequalities, subdivisions of these.
There also exist facet-defining inequalities whose support graph is not complete, see, e.g., \cite[Theorem~2.3]{onthecutpolytope}.

For complete graphs, further classes of facet-defining inequalities are given in \cite{newclassesoffacets,geometryofcuts}.
In particular, for $n \leq 7$  all facets of $\CutP(K_n)$ are classified \cite[Chapter 30.6]{geometryofcuts} and all facets of $\CutP(K_8)$ have been computed \cite[Section~8.3]{DecompositionAndParallelization}, \cite{EnumerationOfFacets}.
It is a major open problems 
to determine the facets of $\CutP(K_n)$ for $n \geq 9$.

\section{$K_{3,3}$-minor-free Graphs}\label{K33-minor-free Graphs}

In this section, we consider $K_{3,3}$-minor-free graphs and provide the complete linear description of their cut polytopes. 
We also show that this yields an efficient algorithm for \MC on $K_{3,3}$-minor-free graphs.
This complements the known facts on $K_5$-minor-free graphs.
Moreover, since $K_5$ is maximal $K_{3,3}$-minor-free but not weakly bipartite, we obtain the first full polyhedral description of a general minor-closed graph class apart from weakly bipartite graphs.

We first characterize maximal $K_{3,3}$-minor-free graphs.
Per se, this is not new: it is sometimes referenced to (different papers by) Wagner; a complete proof in modern terminology was given in \cite{Thomas_excluded}.
Here, we propose a slightly different approach, using $3$-connectivity components.
This provides a simpler, more basic proof and turns out to be directly usable for our polyhedral and our algorithmic results.

Let $G=(V,E)$ be a $2$-connected, not necessarily simple graph and let $\{v,w\}$ be a \emph{split pair} in $G$, i.e., $G-\{v,w\}$ is disconnected or there are parallel edges connecting $v$ and $w$.
The \emph{split classes} of $\{v,w\}$ are given by a partition $E_1, \dots ,E_k$ of $E$ such that two edges are in a common split class if and only if there is a path between them neither containing $v$ nor $w$ as an internal node. As $G$ is $2$-connected, it is easy to see that $v$ and $w$ are both incident to each split class. 
For a split class $C$ let $\overline{C}=E\setminus C$.
A \emph{Tutte split} replaces $G$ by the two graphs $G_1=(V(C),C \cup \{e\})$ and $G_2=(V(\overline{C}), \overline{C}\cup \{e\})$, provided that $G_1-e$ or $G_2-e$ remains $2$-connected.
Thereby, $e$ is a new \emph{virtual} edge connecting $v$ and $w$; the other edges are called \emph{original}.
Observe that this operation may yield parallel edges.
Iteratively splitting the graphs via Tutte splits gives the unique \emph{$3$-connectivity decomposition} 
of $G$. Its components can be partitioned into the following sets: a set $S$ of cycles, a set $P$ of edge bundles (two nodes joined by at least $3$ edges), and a set $R$ of $3$-connected graphs, see, e.g., \cite{tutte,triconnected}.

\begin{lem}\label{K33free_2connected}
Any maximal $K_{3,3}$-minor-free graph $G$ is $2$-connected.
\end{lem}

\begin{proof}
Clearly, $G$ is connected, as otherwise we could join two connected components via an edge without obtaining a $K_{3,3}$-minor.
Assume that $G$ is not $2$-connected and let $v \in G$ be a cut-node separating $G$ into $G_1$ and~$G_2$. 
Choose $w_1 \in G_1$ and $w_2 \in G_2$ adjacent to $v$ and obtain the graph $\tilde{G}$ from $G$ by adding the edge $w_1 w_2$.
As a sidenote, this operation retains planarity for planar $G$.
Since $\tilde{G}$ contains only two paths between $G_1$ and $G_2$ but $K_{3,3}$ is $3$-connected, 
$\tilde{G}$ is still $K_{3,3}$-minor-free. This contradiction concludes the proof.
\end{proof}

\begin{prop}\label{characterizationK33minorfree}
Let $G$ be a maximal $K_{3,3}$-minor-free graph.
Then, $G$ can be decomposed as a strict clique-sum $G= G_1 \oplus_2 \dots \oplus_2 G_k$, where each $G_i$ is either a planar triangulation or a copy of $K_5$.
\end{prop}

\begin{proof}
Let $G$ be a maximal $K_{3,3}$-minor-free graph. By \Cref{K33free_2connected}, $G$ is $2$-connected, so we may consider its $3$-connectivity decomposition.  
Whenever a virtual edge $ab$ was introduced, both parts of the Tutte split contain a path between $a$ and $b$.
Furthermore, $G$ contains a $K_{3,3}$-minor if and only if one of the components of its decomposition does.
But then, if $G$ would not contain an original edge connecting $a$ and $b$, we could introduce it without creating a $K_{3,3}$-minor.
Thus, each virtual edge corresponds to an edge $e \in E(G)$ by maximality of $G$, and
$G$ is the strict $2$-sum of cycles and $3$-connected graphs.
By maximality of $G$, the cycles are triangles, a trivial form of a planar triangulations.

Let $H$ be a $3$-connected graph from this sum.
If $H$ is planar, then~--~by maximality~--~it is a triangulation.
Otherwise, by Kuratowski's Theorem, $H$ contains a $K_5$-subdivision.
Assume that $H \neq K_5$.
If $H$ contains $K_5$ as a subgraph, then it contains the graph shown in \Cref{K33minor_proof}(\subref{K33minor_a}) as a minor, and thus a $K_{3,3}$-minor, which yields a contradiction. 

Assume that $H$ contains a proper $K_5$-subdivision with Kuratowski nodes $S=\{w_1,\dots,w_5\}$ and let $v \in V(H) \setminus S$ be a node of this subdivision.
Since $H$ is 3-connected, there are disjoint paths from $v$ to three pairwise distinct Kuratowski nodes, say $w_1, w_2, w_3$. 
But then $H$ contains the graph of \Cref{K33minor_proof}(\subref{K33minor_b}) as a minor, which itself contains a $K_{3,3}$-minor. This concludes the proof.
\end{proof}

\begin{figure}
\begin{subfigure}{.33\textwidth}\centering
\begin{tikzpicture}[scale=.75]
	\path[use as bounding box] (-1.7,-2) rectangle (4,1.8);
	\node[circle,draw, scale=0.5] (v1) at (360/5:1.8cm) {};
	\node[circle,fill, scale=0.5] (v2) at (2*360/5:1.8cm) {};
	\node[circle,fill, scale=0.5] (v3) at (3*360/5:1.8cm) {};
	\node[circle,draw, scale=0.5] (v4) at (4*360/5:1.8cm) {};
	\node[circle,draw, scale=0.5] (v5) at (5*360/5:1.8cm) {};
\foreach \phi in  {1,...,5}{
	\foreach \psi in {1,...,\phi}  {       
      \draw (v\phi) -- (v\psi);}
      }
	\node[circle,fill, scale=0.5, right of=v5, node distance =2cm](a) {};
	\draw (a)--(v1);
	\draw (a)--(v4);
	\draw (a)--(v5);
\end{tikzpicture}
\caption{}
\label{K33minor_a}
\end{subfigure}%
\begin{subfigure}{.33\textwidth}\centering
\begin{tikzpicture}[every node/.style={circle, draw, scale=.5, fill=white},scale=.75]
	\path[use as bounding box] (-2.1,-2) rectangle (2,1.8);
	\node[label={[label distance=-6]50:\Large{$w_1$}}](v1) at (360/5:1.8cm) {};
	\node[fill=black,label={[label distance=0pt]90:\Large{$v$}}] (v6) at ($0.5*(v1)+0.5*(v2)$) {};
	\node[label={[label distance=-6pt]120:\Large{$w_2$}}] (v2) at (2*360/5:1.8cm) {};
	\node[label={[label distance=0pt]200:\Large{$w_3$}}] (v3) at (3*360/5:1.8cm) {};
	\node[fill=black,label={[label distance=-4pt]-80:\Large{$w_4$}}] (v4) at (4*360/5:1.8cm) {};	 
	\node[fill=black,label={[label distance=-4pt]0:\Large{$w_5$}}] (v5) at (5*360/5:1.8cm) {};
			 
	\draw (v1)--(v6);
	\draw (v1)--(v3);
	\draw (v1)--(v4);
	\draw (v1)--(v5);
	\draw (v2)--(v3);
	\draw (v2)--(v4);
	\draw (v2)--(v5);
	\draw (v2)--(v6);
	\draw (v3)--(v5);
	\draw (v3)--(v4);%
	\draw (v4)--(v5);
	\draw (v6) edge[decorate,decoration={snake,segment length=1mm,amplitude=.3mm}, out=150, in=150, distance=25mm](v3);
	 
\end{tikzpicture}
\caption{}
\label{K33minor_b}
\end{subfigure}%
%
%
%
\caption{Graphs of the proof of \Cref{characterizationK33minorfree}}\label{K33minor_proof}
\end{figure}

\Cref{characterizationK33minorfree} allows us to classify all facets of cut polytopes of maximal $K_{3,3}$-minor-free graphs:

\begin{thm}\label{facetsK33minorfree}
Let $G$ be a maximal $K_{3,3}$-minor-free graph. Then all facets of $\CutP(G)$ are given by the metric inequalities for each triangle contained in $G$ and switchings of the facet-defining inequality 
\begin{equation}\label{K5eq}
\sum_{e \in E(K_5)} x_e \leq 6 \qquad\text{ for each }K_5\text{-subgraph.} 
\end{equation}
\end{thm}
\begin{proof}
We know from \Cref{facetsksum} that the facets of the cut polytope of a $2$-sum of graphs are given by taking all facets of the cut polytopes of both graphs and identifying common variables. Moreover, by \Cref{facetsK5minorfree} all facets of a planar triangulation are given by metric inequalities; the facets of $\CutP( K_5)$ are given by metric inequalities and switchings of (\ref{K5eq})~\cite[Chapter 30.6]{geometryofcuts}. Since maximal $K_{3,3}$-minor-free graphs are $2$-sums of copies of $K_5$ and planar triangulations, this yields the claimed result.
\end{proof}

We can use \Cref{facetsK33minorfree} to classify the facets of the cut polytope of \emph{any} $K_{3,3}$-minor-free graph. 
\begin{cor}
Let $G$ be a $K_{3,3}$-minor-free graph.
Then, $G$ can be decomposed as a (not necessarily strict) k-sum of planar graphs and/or copies of $K_5$, with $k=1,2$.

Let $H$ be a maximal $K_{3,3}$-minor-free graph containing $G$.
Then, the facets of $\CutP(G)$ are obtained by projecting $\CutP(H)$ onto $\{x_e=0~:~ e \in E(H)\setminus E(G) \}$.
\end{cor}
\begin{proof}
The decomposition claim follows from \Cref{characterizationK33minorfree}.
Alternatively, we can obtain $G$ from $H$ by deleting edges. On the level of cut polytopes, the effect of an edge deletion $e \in E(H)$ corresponds to a projection onto $\{x \in \R^E : x_e=0\}$.
\end{proof}

On the level of facets, a projection of a polytope to a coordinate hyperplane is given by eliminating variables. This can be done by Fourier-Motzkin elimination \cite[Chapter 1.2]{ziegler}, which is made more precise in the following example.

\begin{ex}\label{exampleK33}
Consider the non-maximal $K_{3,3}$-minor-free graph $G$ shown in \Cref{example_projection}. It is obtained by taking the non-strict $2$-sum of two copies of $K_5$. Let these copies of $K_5$ be $G_1=(V,E)$ and $G_2=(W,F)$ with $V=\{v_1,v_2,v_3,u_1,u_2 \}$ and $W=\{u_1,u_2,w_1,w_2,w_3 \}$.
\begin{figure}[h]
\centering
\begin{tikzpicture}[every node/.style={circle, draw,minimum size=4mm}]
	\node (w1) at (360/5:1.8cm) {$w_1$};
	\node (w2) at (360/5*2:1.8cm) {$u_2$};
	\node (w3) at (360/5*3:1.8cm) {$u_1$};
	\node (w4) at (360/5*4:1.8cm) {$w_3$};
	\node (w5) at (360/5*5:1.8cm) {$w_2$};
    \draw (w1)--(w2);
    \draw (w1)--(w3);
    \draw (w1)--(w4);
    \draw (w1)--(w5);
    \draw (w2)--(w4);
    \draw (w2)--(w5);
    \draw (w3)--(w4);
    \draw (w3)--(w5);
    \draw (w4)--(w5);
      
 \node[draw=white] (y) at ($(w2)+(w3)$) {};
\begin{scope}[yscale=1,xscale=-1]
	\node (v1) at ($(y)+(360/5:1.8cm)$)   {$v_1$};
	\node (v4) at ($(y)+(360/5*4:1.8cm)$) {$v_3$};
	\node (v5) at ($(y)+(360/5*5:1.8cm)$) {$v_2$};
	\draw (v1)--(w2);
    \draw (v1)--(w3);
    \draw (v1)--(v4);
    \draw (v1)--(v5);
    \draw (w2)--(v4);
    \draw (w2)--(v5);
    \draw (w3)--(v4);
    \draw (w3)--(v5);
    \draw (v4)--(v5);
\end{scope}
\end{tikzpicture}
\caption{Graph of \Cref{exampleK33}}
\label{example_projection}
\end{figure}

Both $G_1-u_1u_2$ and $G_2-u_1u_2$ are planar and each edge is contained in a triangle.
Thus, all facets of their cut polytopes are given by metric inequalities, and those are also facets of $\CutP(G)$.
All other facets of $\CutP(G)$ are obtained by taking a pair of facets ${\mathfrak f}_1$ of $\CutP(G_1)$ and ${\mathfrak f}_2$ of $\CutP(G_2)$ and eliminating the variable $x_{u_1u_2}$ by summing the corresponding inequalities. 
In the following we focus on the latter class of facets.
Choosing one representative for each class of facet-defining inequalities of $G_1$ and $G_2$ we get:
\\
\newcommand{\mytab}{\tabto{7cm}}
\begin{enumerate}[(1)]
\item one metric inequality of $G_1$:      \mytab $x_{u_1u_2} +x_{u_2v_1}+x_{u_1v_1} \leq 2$, 
\\[-2ex]
\item one hypermetric inequality of $G_1$: \mytab $\sum_{e \in E}x_e \leq 6$,                
\\[-2ex]
\item one metric inequality of $G_2$:      \mytab $-x_{u_1u_2}-x_{u_2w_1}+x_{u_1w_1} \leq 0$,
\\[-2ex]
\item one hypermetric inequality of $G_2$: \mytab $ \sum_{f \in F:\ u_2 \notin F} x_f - \sum_{f \in F:\ u_2 \in f}x_f \leq 2$.
\end{enumerate}

Using Fourier-Motzkin elimination we have to sum each pair of inequalities such that there is one facet of each graph:
\begin{align*}
x_{u_2v_1}+x_{u_1v_1}-x_ {u_2w_1}+x_{u_1w_1} & \leq 2, \tag{1+3}\\
\sum_{f \in F:\ u_2 \notin f} x_f -x_{u_2w_1}-x_{u_2w_2} -x_{u_2w_3} +x_{u_2v_1}+x_{u_1v_1} &\leq 4,		\tag{1+4}\\
\sum_{e \in E:\ e \neq u_1u_2} x_e  -x_{u_2w_1}+x_{u_1w_1} &\leq 6,	\tag{2+3}\\
\sum_{e \in E:\ e \neq u_1u_2} x_e+ \sum_{f \in F:\ u_2 \notin f} x_f -x_{u_2w_1}-x_{u_2w_2}-x_{u_2w_3} &\leq 8. \tag{2+4}
\end{align*}
(1+3) is a cycle inequality. Switching  (1+4) at $\delta(\{u_2\})$ shows that this inequality is equivalent to (2+3). These inequalities correspond to  copies of $K_5$ with one subdivided edge contained in $G$. The support graph of facet (2+4) is $G$: This type of inequality is  neither facet-defining for complete graphs nor does it belong to one of the mentioned classes of facet-defining inequalities in \Cref{Preliminaries}. \hfill$\blacktriangleleft$
\end{ex}

As demonstrated in the above example, Fourier-Motzkin elimination yields all facet-defining inequalities of a 
non-maximal $K_{3,3}$-minor-free graph as sums of metric inequalities and hypermetric $K_5$-inequalities.
The \emph{support graph} of a valid inequality $\mathfrak f$ for $\CutP(G)$ is the graph $H\subseteq G$ induced by edges with non-zero coefficients in $\mathfrak f$.
From \Cref{facetsK5minorfree} we can thus deduce that the support graph of a facet-defining inequality is an edge, a cycle or contains a $K_5$-minor.
Considering the sum of two facets ${\mathfrak f}_1$ and ${\mathfrak f}_2$ used to eliminate the variable $x_e$ we observe the following:
If ${\mathfrak f}_2$ is a cycle-inequality, summing it to ${\mathfrak f}_1$ acts on the support graph of ${\mathfrak f}_1$ as subdividing~$e$;
the effect of subdividing an edge in the support graph of a facet is described in \cite[Corollary~2.10]{onthecutpolytope}.
If ${\mathfrak f}_2$ is a hypermetric $K_5$-inequality, summing it to ${\mathfrak f}_1$ acts on the support graph of ${\mathfrak f}_1$ as replacing $e$ by $K_5-e$; 
all non-zero coefficients of the obtained inequality are $\pm 1$. 
Although possible, it is tedious to determine the exact signs and thus the constant term of the inequalities. However, we can concisely describe the facets' support graphs.
\begin{cor}
Let $\mathfrak f$ be a facet of the cut polytope of a $K_{3,3}$-minor-free graph $G$.
All its non-zero coefficients are $\pm1$ and
its support graph is an induced subgraph of $G$ that is either
\begin{itemize}
\item an edge that is not contained in a triangle, or
\item obtained from a triangle or a $K_5$ by repeatedly (possibly zero times) subdividing edges and/or replacing an edge $e$ by $K_5-e$.
\end{itemize}
\end{cor}

\mypar{Algorithmic Consequences}
Barahona \cite[Section 4]{graphsnotcontractible} gave an $\mathcal{O}(|V|^4)$ algorithm for \MC on $K_5$-minor-free graphs. Complementing this result an algorithm for \MC on $K_{3,3}$-minor-free graphs whose running time is identical to that of planar \MC is given. Currently, the best known running time for this is $\mathcal{O}(|V|^\frac{3}{2} \log |V|)$ \cite{MaxCutPlanar2,MaxCutPlanar1}.
We use a data structure to efficiently consider the components of the $3$-connectivity decomposition of~$G$.
Recall that they are cycles $S$, edge bundles $P$, and $3$-connected graphs $R$.
The \emph{SPR-tree} $T=T(G)$ has a node for each element of $S$, $P$, and~$R$~\cite{dibatt,spr}\footnote{The data structure is also known as \emph{SPQR-tree}. However, the originally proposed nodes of type $Q$ (as well as the tree's orientation) have often turned out to be superfluous.}. 
For a node $v \in V(T)$, let $H_v$ denote its corresponding component. 
Two nodes $v,w \in V(T)$ are adjacent if and only if $H_v$ and $H_w$ share a virtual edge. 
$G$ can be reconstructed from $T$ by taking the non-strict $2$-sum of components whenever their corresponding nodes are adjacent in $T$. 
Following this interpretation, $P$-nodes containing a non-virtual edge represent strict $2$-sums of their adjacent components of the decomposition. 
$T$ has only linear size and can be computed in $\mathcal{O}(|E(G)|)$ time~\cite[Lemma 15]{triconnected}.

\begin{thm}
The \MC problem on $K_{3,3}$-minor-free graphs can be solved in the same time complexity as \MC on planar graphs.
\end{thm}
\begin{proof}
Let $G=(V,E)$ be a $K_{3,3}$-minor-free graph with edge weights $c_e$, $e \in E$. 
Let $p(n) \in \Omega(n)$ be the best known running time for \MC on planar graphs with $n$ nodes.
For $A,B \subseteq E$ we denote by $\beta_G(A,B)$ the maximum weight over cuts $\delta\subseteq E(G)$ with $A \subseteq \delta$ and $B \cap \delta = \emptyset$.
If $G$ is not $2$-connected, we apply the algorithm to its $2$-connected components (which can be identified in linear time). 
Assume in the following that $G$ is $2$-connected. 

We want to insert ``original'' edges of weight $0$ into $G$ between split pairs corresponding to Tutte splits.
This will allow us to only consider strict $2$-sums. To this end compute the SPR-tree $T=T(G)$.
For any $P$-node $v \in V(T)$ whose $H_v$ contains only virtual edges, introduce a new original edge of weight $0$ into $H_v$, and therefore also into $G$.
For any adjacent non-$P$-nodes $v,w \in V(T)$, let $ab$ be the virtual edge shared between their components.
We introduce a new original edge $ab$ into $G$. This yields a new $P$-node $u$ subdividing the edge $vw$ in $T$.
The edge bundle $H_u$ contains the new original edge together with two virtual edges, one shared with $H_v$, the other with $H_w$.
By this construction, for every virtual edge there is an original edge with the same end nodes.
Throughout the following, we always consider the weight of a virtual edge $ab$ to be identical to the weight of the original edge $ab$.
We continue to denote the resulting graph and tree by $G$ and $T$, respectively. 

Let $v$ be a leaf in $T$ and $ab$ be the virtual edge contained in $H=H_v$.
Note that $v$ is either an $S$- or an $R$-node and thus, $H$ is either a copy of $K_5$ or planar.
We compute $ \beta^+=\beta_H(\{ab\}, \emptyset)$ and $\beta^-=\beta_H(\emptyset, \{ab\})$. 
If $H=K_5$, this requires only constant time. Thus the needed work is bounded by $\mathcal{O}(p(|(V(H)|))$. 
Let $\gamma = \beta^+-\beta^-$ be the gain/loss by having $ab$ in the cut, respectively.
Removing $v$ from $T$ and therefore all edges of $H_v$ from $G$ yields a graph $G'$.
$T(G')$ is obtained from $T-v$ by removing the potential $P$-node-leaf (and considering the ``dangling'' virtual edge as original, retaining its current cost).
Setting the cost of the \emph{original} edge $ab$ to $\gamma$ (after the computation of $\beta^+$ and $\beta^-$) yields that the maximum cut on $G$ is exactly $\beta^-+\xi$, where $\xi$ is the maximum cut in $G'$ (after updating the edge weight).

In this way, we can iteratively compute a maximum cut on $G$ by eliminating all nodes of its SPR-tree.
The SPR-tree of $G$ can be built in $\mathcal{O}(|E|)$ time.
Let $H_1, \dots , H_k$ be the components corresponding to $R$- and $S$-nodes in $T(G)$, $k \leq |V|$. 
By planarity (or constant size of $K_5$), we have $|E(H_i)| \in \mathcal{O}(|V(H_i|))$, and hence $|E| \in \mathcal{O}(|V|)$.
For each $H_i$, $i \in [k]$, we require only $\mathcal{O}(p(|V(H_i)|))$ time.
Since $p(|V|) \in \Omega(|V|)$ we have $\sum_{i=1}^k p(|V(H_i)|)\in \mathcal{O}(p(|V|))$. The claim follows.
\end{proof}

\section{Simple and Simplicial Cut Polytopes}\label{Simple and Simplicial Cut Polytopes}
In this section, we completely characterize graphs whose cut polytopes are simple or simplicial. 

In \cite{cutpolytopesimple}, it was claimed that $\CutP(G)$ is simple if and only if $G$ contains no $C_4$-minor. Unfortunately, the given proof has some gaps. For example, \cite[Proposition~3.2.4.]{cutpolytopesimple} claims that a $0$-$1$-polytope is simple if and only if it is smooth. The  proof mistakenly assumes that $\CutP(G)$ is always the polytope corresponding to the cut-variety in the sense of toric geometry. It is then used that a toric variety is smooth if and only if the corresponding polytope is, see \cite[Theorem 2.4.3]{ToricVarieties}.
However, the cut polytope $\CutP(K_3)$ is simple but not smooth, since the edges $(1,1,0)$, $(1,0,1)$ and $(0,1,1)$ do not form a basis of $\mathbb{Z}^3$. Contrarily the cut variety of $K_3$ is smooth, see \cite[Corollary 2.4]{ToricGeometyCutsAndSplits}.

Nevertheless, in the following we show that the claimed characterization of graphs whose cut polytopes are simple 
is true.
Our proof only requires basic tools from graph theory and discrete geometry.
\begin{defi}
An \emph{ear} in a graph $G$ is a maximal path whose internal nodes have degree $2$ in $G$. An \emph{ear decomposition} of a $2$-connected graph $G$ is a decomposition $G=\bigcup_{i=0}^n G_i$ such that $G_0$ is a cycle and $G_k$ is an ear of $\bigcup_{i=0}^k G_i$ for all $1 \leq k \leq n$.
\end{defi}
A graph is $2$-connected if and only if it admits an ear decomposition, see, e.g., \cite[Proposition 3.1.2]{diestel}. This allows us to prove the following:
\begin{lem}\label{C4minorfree}
Let $G$ be a connected graph. Then the following are equivalent:
\begin{enumerate}[(i)]
\item $G$ is $C_4$-minor-free;
\item $G=G_1 \oplus_1 \dots \oplus_1 G_k$ with $G_i=K_2$ or $G_i=K_3$ for each $i \in [k]$.
\end{enumerate}
\end{lem}

\begin{proof}
Since $K_2$ and $K_3$ are $C_4$-minor-free and $1$-sums create cut-nodes, it is easy to see that $(ii)$ implies $(i)$. To show the reverse direction, let $G$ be a $C_4$-minor-free graph. Considering its $2$-connected components gives a decomposition $G=G_1 \oplus_1 \dots \oplus_1 G_k$, where $G_i=K_2$ or $G_i$ is $2$-connected.

It is left to show that the only $2$-connected $C_4$-minor-free graph is $K_3$.
Assume that $G$ is a $2$-connected $C_4$-minor-free graph and consider its ear-decomposition $G=G_0 \cup \dots \cup G_k$. Since $G$ is $C_4$-minor-free, $G_0$ is a copy of $K_3$. Attaching an ear to two of its nodes would yield a $C_4$-minor. Hence $G=K_3$.
\end{proof}

Given this characterization, we are able to show that $C_4$-minor-free graphs are exactly those graphs whose cut polytopes are simple.

\begin{thm}
The following are equivalent:
\begin{enumerate}[(i)]
\item $\CutP(G)$ is simple;
\item $G$ is $C_4$-minor-free.
\end{enumerate}
\end{thm}

\begin{proof}

If $G$ is not connected, then $\CutP(G)$ is the product of the cut polytopes of the connected components of $G$. Since the product of polytopes is simple if and only if each of the polytopes is simple, it suffices to show the equivalence for connected graphs.

If $G$ is not $2$-connected, it can be decomposed as ${G=G_1 \oplus _1 ...\oplus _1 G_k}$ such that $G_i$ is either $2$-connected or a copy of $K_2$. Hence $\CutP(G)= \CutP(G_1) \times ... \times \CutP(G_k)$ is simple if and only if $\CutP(G_i)$ is simple for all $i \in [k]$.

We show that for $2$-connected graphs, the cut polytope is simple if and only if the graph equals $K_3$.
As $\CutP(K_2)$ is simple this fact together with \Cref{C4minorfree} yields the claim.

Observe that $\CutP(K_3)$ is a simplex. It hence remains to show that $\CutP(G)$ is not simple in case that $G$ is $2$-connected  and $G \neq K_3$.  
In this case, each edge $e \in E(G)$ is contained in a cycle and thus in particular in a chordless cycle $C_e$. By \Cref{facetsK5minorfree} the inequalities
\begin{equation}\label{cycleineq}
x_e - \sum_{f \in E(C_e) \setminus \{e\}} x_f \leq 0
\end{equation}
define $|E(G)|$ many different facets of $\CutP(G)$ that contain the origin.

If $G=C_n$, $n \geq 4$, no edge is contained in a triangle and $ x_e \geq 0$ defines a facet of $\CutP(G)$ for all $e \in E$. Hence, $0$ is contained in at least $2|E(G)|$ many different facets and as $\dim(\CutP(G)) = |E(G)|$, the cut polytope 
is not simple.

Similarly, if $G \neq C_n$, then there has to exist a chord $e$ in some cycle. In particular, $e$ lies in two chordless cycles. Thus,  the origin is contained in at least $|E|+1$ facets and hence $\CutP(G)$ is not simple.
\end{proof}

Next we study graphs whose cut polytopes are simplicial.
It was shown in \cite{geometryofcuts} that the cut polytope of $K_n$ is not simplicial for $n \geq 5$. We generalize this result by giving a complete characterization of graphs with simplicial cut polytopes. We start by proving that the cut polytope of an \emph{arbitrary} graph on at least $5$ nodes is not simplicial.
In addition to this, \Cref{tabularsimplicial} lists all graphs on at most four non-isolated nodes and indicates whether their cut polytopes are simplicial or not.

\tikzset{every path/.style={},
graph/.pic={
\node[circle,fill, scale=0.2] (a) at (0,0) {};
\node[circle,fill, scale=0.2] (b) at (.3,0) {};
\node[circle,fill, scale=0.2] (c) at (.3,.3) {};
\node[circle,fill, scale=0.2] (d) at (0,.3) {};
}}
\tikzset{every path/.style={},
graph1/.pic={
\node[circle,fill, scale=0.2] (a) at (0,0) {};
\node[circle,fill, scale=0.2] (b) at (.3,0) {};
}}
\tikzset{every path/.style={},
graph2/.pic={
\node[circle,fill, scale=0.2] (a) at (0,0) {};
\node[circle,fill, scale=0.2] (b) at (.3,0) {};
\node[circle,fill, scale=0.2] (d) at (0,.3) {};
}}

\begin{table}[ht]
\begin{center}
\caption{All graphs on $n \leq 4$ non-isolated nodes (cf. \Cref{Simplicial_characterization})}\label{tabularsimplicial}
\renewcommand{\arraystretch}{1.3}
\begin{tabular}{| c | c | c | c | c | c | c | c | c | c | c |}

\hline
Graph $G$ &
	\tikz{\pic {graph1}; \draw (a)--(b);}	&
	\tikz[baseline=1]{\pic {graph2}; \draw (a)--(b); \draw (a)--(d);}	&
	\tikz[baseline=1]{\pic {graph2}; \draw (a)--(b);\draw (a)--(d); \draw (b)--(d);}	&
	\tikz[baseline=1]{\pic {graph}; \draw (a)--(b); \draw (b)--(c); \draw (a)--(d);  }	&
	\tikz[baseline=1]{\pic {graph}; \draw (a)--(b); \draw (a)--(c); \draw (a)--(d);}	&
	\tikz[baseline=1]{\pic {graph}; \draw (a)--(b);  \draw (c)--(d); }	&
	\tikz[baseline=1]{\pic {graph};\draw (a)--(b); \draw (b)--(c); \draw (c)--(d); \draw(d)--(a); }	&
	\tikz[baseline=1]{\pic {graph}; \draw (a)--(b);\draw (a)--(d); \draw (b)--(d); \draw (b)--(c);}	&
	\tikz[baseline=1]{\pic {graph};\draw (a)--(b);\draw (a)--(d); \draw (b)--(d); \draw(c)--(d); \draw (c)--(b); }	&
	\tikz[baseline=1]{\pic {graph}; \draw (a)--(b);\draw (a)--(d); \draw (b)--(d);\draw(c)--(d); \draw (c)--(b); \draw (a)--(c);}	
\\ \hline
 $\CutP (G)$ simplicial?	&  \cmark	&\cmark		&\cmark		&\xmark	& \xmark		&	\cmark	&	\cmark	&	\xmark	&	\xmark	&	\cmark	 \\ \hline
Proof case: &		&		&	&	(a)	&	(a)	&	&	&	(b)	&	(b)	& \\ \hline

\end{tabular}
\renewcommand{\arraystretch}{1}
\end{center}
\end{table}

\begin{prop}\label{simplicial_5vertices}
Let $G=(V,E)$ with $|V|=n \geq 5$ without isolated nodes. Then $\CutP(G)$ is not simplicial.
\end{prop}

\begin{proof}
First note that the product $P \times Q$ of polytopes $P$ and $Q$ is simplicial if and only if either $\dim(P)=\dim(Q)=1$ or one is simplicial and the other one is a point.
The latter case is not relevant for us, since a graph with edges has a cut polytope of dimension at least $1$.

Assume that $G$ is disconnected with connected components $G_1, \dots, G_r$. Then $\CutP(G)=\CutP(G_1) \times \dots \times \CutP(G_r)$ is simplicial if and only if $r=2$ and $\CutP(G_1)$ and $\CutP(G_2)$ are $1$-dimensional. This holds if and only if $G_1$ and $G_2$ are copies of $K_2$. Hence, if $G$ is disconnected, $\CutP(G)$ is not simplicial for $n \geq 5$.

If $G$ is connected, we consider the following two cases:

\emph{(a) $G$ contains no triangle:} Let $e=uv \in E$ be an edge of $G$. As, by assumption, $e$ is not contained in a triangle, $x_e \geq0$ defines a facet of $\CutP(G)$. For each ${S \subseteq V \setminus\{u,v\}}$, the indicator vector of the cut $\delta(S \cup \{u,v\})$ is a vertex of this facet. Hence, this facet contains at least $2^{n-2}$ vertices. On the other hand, since $G$ contains no triangle, Turáns Theorem (see \cite[Theorem 7.1.1.]{diestel}) yields $\dim(\CutP(G))=|E| \leq \lfloor \frac{n^2}{4}  \rfloor$. Since $n \geq 5$, we have $2^{n-2} > \frac{n^2}{4}$ which implies that $\CutP(G)$ is not simplicial.

\emph{(b) $G$ contains a triangle:} Let $\Delta=(W,F)$ with $W=\{v_1,v_2,v_3\}$ and $F=\{e,f,g\}$ be a triangle in $G$. By \Cref{facetsK5minorfree}, the inequality $x_e+x_f+x_g \leq 2$ is facet-defining for $\CutP(G)$. For each $S \subseteq V \setminus \{v_1,v_2,v_3\}$ and $i \in \{1,2,3\}$, the indicator vector of the cut $\delta(S \cup \{v_i\})$ is a vertex of this facet. This gives $3 \cdot 2^{n-3}$ vertices on this facet. As $n\geq 5$, we have $3 \cdot 2^{n-3} >\frac{n(n-1)}{2} \geq |E|$. Hence, $\CutP(G)$ is not simplicial.
\end{proof}

Using the previous proposition, we are able to give a characterization of all graphs whose cut polytopes are simplicial (see also \Cref{tabularsimplicial}).

\begin{thm}\label{Simplicial_characterization}
Let $G$ be a graph with no isolated nodes. Then the following are equivalent:
\begin{enumerate}[(i)]
\item $\CutP(G)$ is simplicial;
\item $G$ is one of the following graphs:
	$$K_2,\ K_2 \dotcup K_2,\ K_2 \oplus_1 K_2,\ K_3,\ K_4 \text{ or } C_4. $$
\end{enumerate}
\end{thm}

\begin{proof}
Note that $\CutP(K_2)$ is $1$-dimensional, hence simplicial. By \Cref{CutPolytopeDisjoint} and \Cref{CutPolytope1Sum}, this yields that $\CutP(K_2 \dotcup K_2)$ and $\CutP(K_2 \oplus_1 K_2)$ are simplicial.
It is straight-forward to verify that $\CutP(K_3)$ is a 3-simplex, $\CutP(K_4)$ is affine isomorphic to the $6$-dimensional cyclic polytope on $8$ vertices, and $\CutP(C_4)$ is affinely isomorphic to a cross-polytope, all of which are simplicial polytopes.

By \Cref{simplicial_5vertices}, $\CutP(G)$ is not simplicial if $G$ contains more than $4$ nodes. Thus, it is only left to show that the remaining graphs in \Cref{tabularsimplicial} are not simplicial. It follows from cases (a) and (b) of the proof of \Cref{simplicial_5vertices} that the graphs labeled with (a) and (b), respectively, are not simplicial. 
\end{proof}

\section{Conclusion}
We have determined the linear description of cut polytopes of $K_{3,3}$-minor-free graphs and classified all graphs with a simple or simplicial cut polytope.

Throughout this paper one can see that besides graph minors, the decomposition of graphs into clique-sums of specific graphs is a useful tool to understand cut polytopes. This motivates several questions discussed in the following.

In \cite{cutsandcontainment}, it was shown that for each single-crossing graph $H$, \MC can be solved in polynomial time on the class of $H$-minor-free graphs. For $H=K_5$ and $H=K_{3,3}$ the linear description of cut polytopes of $H$-minor-free graphs is now known. This naturally leads to the following question:
\begin{question}
Can one give the linear description of cut polytopes of $H$-minor-free graphs, for single-crossing graphs $H \neq K_5,K_{3,3}$?
\end{question}

By \Cref{facetsksum} for $k\leq 3$ the linear description of a strict $k$-sum of two graphs is given by taking all facet-defining inequalities of both graphs and identifying common variables. This can be traced back to the fact that in these cases $\CutP(K_k)$ is a simplex. Although this does not hold for $k \geq 4$, the cut polytope of $K_4$ is a cyclic polytope and as such well understood. Therefore, the following question arises:
\begin{question}
Can one give a linear description of the $4$-sum of two graphs in terms of their linear descriptions?
\end{question}

While we give a linear description of cut polytopes of $K_{3,3}$-minor-free graphs in \Cref{K33-minor-free Graphs}, there are further graphs that fall under the same facet regime (interestingly, even $K_{3,3}$ itself). We thus ask:

\begin{question}
Can one characterize all graphs whose cut polytopes are described by the inequalities from \Cref{K33-minor-free Graphs}?
\end{question}

\newpage

\bibliography{bibliography}

\begin{thebibliography}{BGM85}

\bibitem[AI07]{newclassesoffacets}
D.~Avis and T.~Ito.
\newblock New classes of facets of the cut polytope and tightness of
  {$I_{mm22}$} {B}ell inequalities.
\newblock {\em Discrete Appl. Math.}, 155(13):1689--1699, 2007.

\bibitem[Bar82]{SpinGlas}
F.\ Barahona.
\newblock On the computational complexity of {I}sing spin glass models.
\newblock {\em J. Phys. A}, 15(10):3241--3253, 1982.

\bibitem[Bar83]{graphsnotcontractible}
F.~Barahona.
\newblock The max-cut problem on graphs not contractible to {$K_{5}$}.
\newblock {\em Oper. Res. Lett.}, 2(3):107--111, 1983.

\bibitem[BG09]{BrunsGubeladze}
W.~Bruns and J.~Gubeladze.
\newblock {\em Polytopes, rings, and {$K$}-theory}.
\newblock Springer Monographs in Mathematics. Springer, Dordrecht, 2009.

\bibitem[BGM85]{bipartitesubgraph}
F.\ Barahona, M.\ Gr\"{o}tschel, and A.R.\ Mahjoub.
\newblock Facets of the bipartite subgraph polytope.
\newblock {\em Math. Oper. Res.}, 10(2):340--358, 1985.

\bibitem[BM86]{onthecutpolytope}
F.~Barahona and A.R. Mahjoub.
\newblock On the cut polytope.
\newblock {\em Math. Programming}, 36(2):157--173, 1986.

\bibitem[BR88]{CircuitLayout}
F.~Barahona and G.~Reinelt.
\newblock An application of combinatorial optimization to statistical physics
  and circuit layout design.
\newblock {\em Operat. Research}, 36:493--513, 1988.

\bibitem[CH17]{spr}
M.~Chimani and P.~Hlin\v{e}n\'{y}.
\newblock A tighter insertion-based approximation of the crossing number.
\newblock {\em J. Comb. Optim.}, 33:1183--1225, 2017.

\bibitem[CLS11]{ToricVarieties}
D.A. Cox, J.B. Little, and H.K. Schenck.
\newblock {\em Toric Varieties}.
\newblock Graduate studies in mathematics. American Mathematical Society, 2011.

\bibitem[CR01]{DecompositionAndParallelization}
T.~Christof and G.~Reinelt.
\newblock Decomposition and parallelization techniques for enumerating the
  facets of combinatorial polytopes.
\newblock {\em Internat. J. Comput. Geom. Appl.}, 11(04):423--437, 2001.

\bibitem[dBT96]{dibatt}
G.~di~Battista and R.~Tamassia.
\newblock On-line planarity testing.
\newblock {\em SIAM J. Comput.}, 25:956--997, 1996.

\bibitem[Die18]{diestel}
R.\ Diestel.
\newblock {\em Graph theory}, volume 173 of {\em Graduate Texts in
  Mathematics}.
\newblock Springer, Berlin, fifth edition, 2018.

\bibitem[DL94a]{ApplicationsCutPolyhedra1}
M.~Deza and M.~Laurent.
\newblock Applications of cut polyhedra {I}.
\newblock {\em J. Comput. Appl. Math.}, 55(2):191--216, 1994.

\bibitem[DL94b]{ApplicationsCutPolyhedra2}
M.~Deza and M.~Laurent.
\newblock Applications of cut polyhedra {II}.
\newblock {\em J. Comput. Appl. Math.}, 55(2):217--247, 1994.

\bibitem[DL10]{geometryofcuts}
M.M.\ Deza and M.\ Laurent.
\newblock {\em Geometry of cuts and metrics}, volume~15 of {\em Algorithms and
  Combinatorics}.
\newblock Springer, Heidelberg, 2010.

\bibitem[DS16]{EnumerationOfFacets}
M.\ Deza and M.D.\ Sikiri\'{c}.
\newblock Enumeration of the facets of cut polytopes over some highly symmetric
  graphs.
\newblock {\em Int. Trans. Oper. Res.}, 23(5):853--860, 2016.

\bibitem[Eng11]{CutIdealsK4MinorFree}
Alexander Engstr\"{o}m.
\newblock Cut ideals of {$K_4$}-minor free graphs are generated by quadrics.
\newblock {\em Michigan Math. J.}, 60(3):705--714, 2011.

\bibitem[FMU92]{WeaklyBipartiteK5}
J.~Fonlupt, A.R. Mahjoub, and J.P. Uhry.
\newblock Compositions in the bipartite subgraph polytope.
\newblock {\em Discrete Mathematics}, 105(1):73 -- 91, 1992.

\bibitem[Gan13]{cutpolytopesimple}
A.~Ganguly.
\newblock Properties of cut polytopes.
\newblock {\em University of Minnesota Digital Conservancy}, 2013.

\bibitem[GP81]{WeaklyBipartite}
M.~Grötschel and W.R. Pulleyblank.
\newblock Weakly bipartite graphs and the max-cut problem.
\newblock {\em Oper. Res. Lett.}, 1(1):23--27, 1981.

\bibitem[Had75]{AlgorithmPlanar2}
F.O. Hadlock.
\newblock Finding a maximum cut of a planar graph in polynomial time.
\newblock {\em SIAM J. Comput.}, 4:221--225, 1975.

\bibitem[HT73]{triconnected}
J.~Hopcroft and R.~Tarjan.
\newblock Dividing a graph into triconnected components.
\newblock {\em SIAM J. Comput.}, 2(3):135--158, 1973.

\bibitem[Kam12]{cutsandcontainment}
M.\ Kami\'{n}ski.
\newblock M{AX}-{CUT} and containment relations in graphs.
\newblock {\em Theoret. Comput. Sci.}, 438:89--95, 2012.

\bibitem[Kar72]{karp}
R.M. Karp.
\newblock Reducibility among combinatorial problems.
\newblock In {\em Proceedings of a symposium on the Complexity of Computer
  Computations}, pages 85--103, 1972.

\bibitem[Kur30]{kuratowski}
C.~Kuratowski.
\newblock Sur le problème des courbes gauches en topologie.
\newblock {\em Fundamenta Mathematicae}, 15(1):271--283, 1930.

\bibitem[LP12]{MaxCutPlanar2}
F.~Liers and G.~Pardella.
\newblock Partitioning planar graphs: A fast combinatorial approach for
  max-cut.
\newblock {\em Computational Optimization and Applications}, 51:323--344, 01
  2012.

\bibitem[OD72]{AlgorithmPlanar1}
G.I. Orlova and Y.G. Dorfman.
\newblock Finding the maximal cut in a graph.
\newblock {\em Cybernetics}, 10:502--504, 1972.

\bibitem[Ohs10]{NormalityOfCutPolytopes}
H.~Ohsugi.
\newblock Normality of cut polytopes of graphs is a minor closed property.
\newblock {\em Discrete Math.}, 310:1160--1166, 2010.

\bibitem[Ohs14]{GorensteinCutPolytopes}
H.~Ohsugi.
\newblock Gorenstein cut polytopes.
\newblock {\em Eur. J. Comb.}, 38:122--129, 2014.

\bibitem[RS18]{RetractsCutAlgebra}
T.~R{\"o}mer and S.~{Saeedi Madani}.
\newblock Retracts and algebraic properties of cut algebras.
\newblock {\em Eur. J. Comb.}, 69:214--236, 2018.

\bibitem[SS08]{ToricGeometyCutsAndSplits}
B.~Sturmfels and S.~Sullivant.
\newblock Toric geometry of cuts and splits.
\newblock {\em Michigan Math. J.}, 57:689--709, 2008.

\bibitem[SWK90]{MaxCutPlanar1}
W.-K. Shih, S.~Wu, and Y.~S. Kuo.
\newblock Unifying maximum cut and minimum cut of a planar graph.
\newblock {\em IEEE Trans. Comput.}, 39(5):694--697, May 1990.

\bibitem[Tho99]{Thomas_excluded}
R.~Thomas.
\newblock Recent excluded minor theorems for graphs.
\newblock In {\em Survey in Combinatorics}, pages 201--222. Univ. Press, 1999.

\bibitem[Tut66]{tutte}
W.T. Tutte.
\newblock {\em Connectivity in Graphs}.
\newblock Univ.\ of Toronto Press, 1966.

\bibitem[Wag37]{WagnerPlanar}
K.~Wagner.
\newblock {\"U}ber eine {E}igenschaft der ebenen {K}omplexe.
\newblock {\em Math. Ann.}, 114(1):570--590, 1937.

\bibitem[Zie12]{ziegler}
G.M. Ziegler.
\newblock {\em Lectures on Polytopes}.
\newblock Graduate Texts in Mathematics. Springer New York, 2012.

\end{thebibliography}
\bibliographystyle{alpha}
\end{document}